\documentclass[preprint,12pt,3p]{elsarticle}

\usepackage{amssymb}
\usepackage{amsthm}
\usepackage{amsmath}
\usepackage{mathtools}
\usepackage{siunitx}

\usepackage[hidelinks,colorlinks=false]{hyperref}
\usepackage{color}
\usepackage{enumitem}
\usepackage[usenames,dvipsnames]{xcolor}
\usepackage{colortbl}
\usepackage{caption}
\usepackage{enumerate} 
\newtheorem{theorem}{Theorem}
\newtheorem{lemma}[theorem]{Lemma}

\newtheorem{corollary}[theorem]{Corollary}
\newtheorem{conjecture}[theorem]{Conjecture}

\newtheorem{question}[theorem]{Question}

\begin{document}

\begin{frontmatter}

\title{On the neighborhood complex of $\vec{s}$-stable Kneser graphs}

\author[label1]{Hamid Reza Daneshpajouh}%\corref{cor1}\fnref{label3}}
\address[label1]{School of Mathematics, Institute for Research in Fundamental Sciences (IPM),
Tehran, Iran, P.O. Box 19395-5746}
%\address[label2]{Address Two\fnref{label4}}

%\cortext[cor1]{I am corresponding author}
%\fntext[label3]{I also want to inform about\ldots}
%\fntext[label4]{Small city}
\ead{hr.daneshpajouh@ipm.ir}

\author[label5]{J\'{o}zsef Oszt\'{e}nyi}
\address[label5]{Department of Natural Sciences and Engineering, Faculty of Mechanical Engineering and Automation, John
von Neumann University, Izsáki út 10, Kecskemét, 6000, Hungary}
\ead{ osztenyi.jozsef@gamf.uni-neumann.hu}

%\author[label1,label5]{Author Three}
%\ead{author.three@mail.com}

\begin{abstract}
In 2002, A. Bj\"orner and M. de Longueville showed the neighborhood complex of the $2$-stable Kneser graph ${KG(n, k)}_{2-\textit{stab}}$ has the same homotopy type as the $(n-2k)$-sphere. A short time ago, an analogous result about the homotopy type of the neighborhood complex of almost $s$-stable Kneser graph has been announced by J. Oszt\'{e}nyi. Combining this result with the famous Lov\'{a}sz's topological lower bound on the chromatic number of graphs has been yielded a new way for determining the chromatic number of these graphs which was determined a bit earlier by P. Chen. 

In this paper we present a common generalization of the mentioned results. We will define the $\vec{s}$-stable Kneser graph ${KG(n, k)}_{\vec{s}-\textit{stab}}$ as the induced subgraph of the Kneser graph $KG(n, k)$ on $\vec{s}$-stable vertices. And we prove, for given an integer vector $\vec{s}=(s_1,\ldots, s_k)$ and $n\geq\sum_{i=1}^{k-1}s_i+2$ where $s_i\geq2$ for $i\neq k$ and $s_k\in\{1,2\}$, the neighborhood complex of ${KG(n, k)}_{\vec{s}-\textit{stab}}$ is homotopy equivalent to the $\left(n-\sum_{i=1}^{k-1}s_i-2\right)$-sphere. In particular, this implies that $\chi\left({KG(n, k)}_{\vec{s}-\textit{stab}}\right)= n-\sum_{i=1}^{k-1}s_i$ for the mentioned parameters. Moreover, as a simple corollary of the previous result, we will determine the chromatic number of 3-stable kneser graphs with at most one error.
\end{abstract}

\end{frontmatter}

\section{Introduction}
Throughout this paper, the symbol $[n]$ stands for the set $\{1,\ldots, n\}$. A subset $A\subset [n]$ is called $s$-stable (almost $s$-stable) if $s\leq |i-j|\leq n-s$ ($s\leq |i-j|$) for each distinct $i, j\in A$. Kneser graph $KG(n, k)$ is a graph whose vertices are all $k$-subsets of the set $[n]$ and two of them are adjacent if their corresponding sets has empty intersection. The induced subgraph of $KG(n, k)$ on (almost) $s$-stable vertices is called (almost) $s$-stable Kneser graph and denoted by ${KG(n, k)}_{s-\textit{stab}}$ (${KG(n, k)}_{\widetilde{s-\textit{stab}}}$).

\subsection{History and Motivation}
In 1978, solving Kneser's conjecture by Lov\'{a}sz~\cite{Lovasz} has surprisingly opened a new door in mathematics: Using topological methods in combinatorics! Kneser, in 1955, raised a question in combinatorics which had remained unsolved for almost two decades. His conjecture, in the setting of graph coloring, was about the chromatic number of the Kneser graphs. Actually, Kneser gave a proper coloring of $KG(n, k)$ with $n-2(k-1)$ number of colors and conjectured that this number of colors is always necessary, that is $\chi (KG(n, k))= n-2(k-1)$ for $n\geq 2k$. Lov\'{a}sz's proof contained two main steps. First, to a given graph $G$, he associated a simplicial complex $\mathcal{N}(G)$, called neighborhood complex, whose simplices are subsets of the vertices of $G$ which have a common neighbor. Next, he showed that if the neighborhood complex $\mathcal{N}(G)$ of $G$ is (topologically) $k$-connected, then the chromatic number of $G$ is at least $k+3$. Finally, Lov\'{a}sz completed his proof by verifying that $\mathcal{N} (KG(n,k))$ is $(n-2k-1)$-connected. Shortly afterward, Schrijver~\cite{Schrijver} found a critical subgraph of $KG(n,k)$, called Schrijver's graph and denoted by $SG(n,k)$,  with the same chromatic number. It is worth pointing out that Schrijver used a different technique for computing the chromatic number of $SG(n,k)$. In particular, his proof did not say anything about the topology of the neighborhood complex of Schrijver's graphs. So, it was a natural question to ask whether one can determine the chromatic number of Schrijver's graphs via studying the neighborhood complex of them. In 2002, A. Bj\"{o}rner and M. de Longueville~\cite{lon} gave a positive answer to this question. Indeed, they showed the neighborhood complex of $SG(n,k)$ has the homotopy type of the $(n-2k)$-sphere. Later, the $s$-stable Kneser graph $KG(n, k)_{s-\textit{stab}}$ as a generalization of Schrijver's graph ($SG(n,k)=KG(n,k)_{2-\textit{stab}}$) was introduced by F. Meunier and he made the following conjecture.

\begin{conjecture}[\cite{Meunier1}]\label{conj}
Let $n,k, s$ be non-negative integers such that $n\geq sk$ and $s\geq 2$. Then
$$\chi\left({KG(n,k)}_{s-stab}\right)= n-s(k-1).$$
\end{conjecture}

This conjecture has been confirmed for all even $s$~\cite{chen2}, and for $s\geq 4$ when $n$ is sufficiently large~\cite{Jak}. But, the conjecture is completely open for $s=3$. A weaker form of this conjecture has also been known to be true. Actually, in 2017, P. Chen~\cite{chen} showed that the same formula is valid for almost s-stable Kneser graphs. Similar to Schrijver's proof, his proof was not based on Lov\'{a}sz's bound. But, very recently, J. Oszt\'{e}nyi has shown that Lov\'{a}sz's bound works for this problem as well~\cite{osz}. Indeed, he showed that the neighborhood complex of $KG(n, k)_{\widetilde{{s-stab}}}$ is homotopy equivalent to the $\mathbb{S}^{n-s(k-1)-2}$ for $n\geq sk$ and $s\geq 2$.
\subsection{Aims and objectives}
The main purpose of this paper is to present a common generalization of the A. Bj\"orner and M. de Longueville result and the J. Oszt\'{e}nyi result. In the first steep, we need to generalize the concept of stability as follows. For a $k$-set $A\subseteq [n]$, let $A(1),A(2),\ldots,A(k)$ be the ordered elements of $A$, that is $A(1)$ is the smallest element and $A(k)$ is the largest element of $A$ in the standard order. If $\vec{s}=(s_1,\ldots, s_k)$ is an integer vector, then a $k$-subset $A\subseteq [n]$ is called \textit{$\vec{s}$-stable}, if $s_j\leq A(j+1)-A(j)$ for $1\leq j \leq k-1$ and $A(k)-A(1)\leq n-s_k$. The \textit{$\vec{s}$-stable Kneser graph} $KG(n, k)_{\vec{s}-\textit{stab}}$ was obtained by restricting the vertex set of $KG(n,k)$ to the $\vec{s}$-stable $k$-subsets. Note that the cases $\vec{s}=(s,\ldots, s)$, and $\vec{s}=(s,\ldots, s,1)$ demonstrate the usual concept of s-stable and almost s-stable, respectively. Now, we are in a position to state our main result.
\begin{theorem}\label{main}
Let $n, k,$ be positive integers and $\vec{s}=(s_1, \ldots, s_k)$ be an integer vector where $k\geq 2$, $n\geq\sum_{i=1}^{k-1}s_i+2$, $s_i\geq2$ for $i\neq k$ and $s_k\in\{1,2\}$. Then, the neighborhood complex of $KG(n, k)_{\vec{s}-\textit{stab}}$ is homotopy equivalent to the $\left(n-\sum_{i=1}^{k-1}s_i-2\right)$-sphere. 
\end{theorem}

Next, in conjunction with Lov\'{a}sz's topological bound on the
chromatic number, we can determine the chromatic number of $\vec{s}$-stable Kneser graph for the mentioned parameters. In particular, this leads to a generalization of the P. Chen result.

\begin{theorem}\label{chromaticnumber}
Let $n, k,$ be positive integers and $\vec{s}=(s_1, \ldots, s_k)$ be an integer vector where $k\geq 2$, $n\geq\sum_{i=1}^{k-1}s_i+2$, $s_i\geq2$ for $i\neq k$ and $s_k\in\{1,2\}$. Then
	\[
	\chi\left({KG(n, k)}_{\vec{s}-\textit{stab}}\right)= n-\sum_{i=1}^{k-1}s_i.
	\] 
\end{theorem}

As a simple corollary of the above theorem we give a lower bound on the chromatic number of 3-stable Kneser graphs with at most one error. 

\subsection{Organization of the paper}
Section 2, presents preliminaries, tools, and a presentation of the problem that is more convenient in our setting. In Section 3, we prove the aforementioned result about the homotopy type of the neighborhood complex. Finally, in Section 4, we study the chromatic number of \textit{$\vec{s}$-stable Kneser graph}.

\section{Preliminaries, tools, and presentation}
The main purpose of this section is to present a simpler description of $\mathcal{N}\left(KG(n, k)_{\vec{s}-\textit{stab}}\right)$. 
Here and in what follows, we assume that the reader is familiar with the general theory of simplicial complex~\cite{Matousek} and homotopy theory~\cite{hatcher}. We just fix our notation and recall some basic facts that are crucial for our purpose.

Throughout this paper, the order complex of a poset $P$ is denoted by $\Delta (P)$. We say that the order-preserving maps $f: P\to Q$ between posets $P, Q$ are homotopic, denoted by $f\simeq g$, if the induced simplicial maps $f, g: \Delta (P)\to \Delta(Q)$ are homotopic. In addition, we say poset $P$ and $Q$ are homotopy equivalent, and denoted by $P\simeq Q$, if $\Delta(P)\simeq \Delta(Q)$. The following theorem, known as the Order Homotopy Theorem, is a very useful tool for showing homotopy equivalences among simplicial complexes.  
\begin{theorem}[\cite{Bjo}]
If $f, g : P\to Q$ are two order-preserving maps of finite posets with $f(x)\preceq g(x)$ for all $x\in P$, then $f\simeq g$. In particular, an order preserving map $\psi : P\to P$ with 
\begin{itemize}
    \item $\psi (x)\succeq x$ for all $x\in P$ (increasing operator) or
    \item $\psi (x)\preceq x$ for all $x\in P$ (decreasing operator),
\end{itemize}
induces a homotopy equivalence between $P$ and $Img(\psi)$.
\end{theorem}

For given positive integers $n, k$, and a positive integer vector $\vec{s}=(s_1,\ldots, s_k)$, let $P(n,k,\vec{s})$ be the poset whose elements are all $2$-tuples $(A,B)$ of subsets of $[n]$ with $A\cap B=\emptyset$ and each of $A$ and $B$ contains $\vec{s}$-stable $k$-set. And, its partial order is defined by $(A_1, B_1) \leq (A_2, B_2)$ if $A_1 \subseteq A_2$ and $B_1\subseteq B_2$. Now, in this setting, a simpler description of $\mathcal{N}\left(KG(n,k)_{\vec{s}-\textit{stab}}\right)$ just means $\mathcal{N}\left(KG(n,k)_{\vec{s}-\textit{stab}}\right)\simeq\Delta(P(n,k,\vec{s}))$ which we will prove it in the rest of this section. To do this, we just need one more definition. 

The neighborhood complex is not the only complex that is used for bounding the chromatic number of graphs! As another example of such complexes, we can mention Hom-complexes which were introduced by Lov\'{a}sz as a generalization of neighborhood complex. In fact, the Hom-complex $Hom(H, G)$ is a complex associated to a pair of graphs $H, G$ and it carries topological information of all homomorphisms between $H$ and $G$. Indeed, one of the main motivation of defining this object is that $Hom(K_2, G)$ is homotopy equivalent to $\mathcal{N}(G)$ for every graph $G$~\cite{Kozlov1}. We just recall the definition of this special case of Hom-complexes, as it is needed for our purpose. The Hom-poset ${Hom}_{p}(K_2, G)$ is a poset whose elements are given by all $(\mathcal{A},\mathcal{B})$ of non-empty disjoint subsets of $V(G)$, such that for every $x\in \mathcal{A}$, $y\in \mathcal{B}$ we have $\{x, y\}\in E(G)$. And, its partial order is defined by $(\mathcal{A}_1,\mathcal{B}_1)\leq (\mathcal{A}_2, \mathcal{B}_2)$ if $\mathcal{A}_1 \subseteq \mathcal{A}_2$ and $\mathcal{B}_1\subseteq \mathcal{B}_2$. Finally, the order complex of $Hom_{p}(K_2, G)$ is called Hom-complex, and denoted by $Hom(K_2, G)$. 
\begin{lemma}\label{neighborhoodcomplex}
Let $n, k$ be positive integer and $\vec{s}=(s_1, \ldots, s_k)$ be an positive vectors. Then, 
$$\mathcal{N}\left(KG(n, k)_{\vec{s}-\textit{stab}}\right)\simeq \Delta\left(P(n, k,\vec{s})\right).$$  
\end{lemma}
\begin{proof}
For simplicity of notation, for any $A\subseteq [n]$ let ${\binom{A}{k}}_{\vec{s}-\textit{stab}}$ denote the set of all $\vec{s}$-stable $k$-subsets of $A$. First, note that $\mathcal{N}\left(KG(n, k)_{\vec{s}-\textit{stab}}\right)=\emptyset$ if and only if $\Delta\left(P(n, k,\vec{s})\right)=\emptyset$. Thus, without loss of generality we can assume that $\mathcal{N}\left(KG(n, k)_{\vec{s}-\textit{stab}}\right)\neq\emptyset$. As $\mathcal{N}(G)$ and $Hom(K_2, G)$ have the same homotopy type for every graph $G$, it is enough to show that $\Delta\left(P(n, k,\vec{s})\right)$ and $Hom\left(K_2, KG(n, k)_{\vec{s}-\textit{stab}}\right)$ are homotopy equivalent.  Define the order preserving maps
 \begin{align*}
\varphi:Hom_p\left(K_2, KG(n, k)_{\vec{s}-\textit{stab}}\right) &\longrightarrow P(n, k, \vec{s})\\
  {\left(\mathcal{A},\mathcal{B}\right)} &\longmapsto  {\left(\bigcup\mathcal{A}, \bigcup\mathcal{B}\right)},
\end{align*}
and
 \begin{align*}
  \psi: P(n, k, \vec{s}) &\longrightarrow Hom_p\left(K_2, KG(n, k)_{\vec{s}-\textit{stab}}\right) \\
  {\left(A,B\right)}&\longmapsto {\left({\binom{A}{k}}_{\vec{s}-\textit{stab}},{\binom{B}{k}}_{\vec{s}-\textit{stab}}\right)}.
\end{align*}
Clearly, 
$$Id_{Hom_p\left(K_2, KG(n,k)_{\vec{s}-\textit{stab}}\right)}((\mathcal{A},\mathcal{B})) \preceq\psi\circ\varphi((\mathcal{A},\mathcal{B}))$$ 
for each $(\mathcal{A},\mathcal{B})\in Hom_p\left(K_2, KG(n, k)_{\vec{s}-\textit{stab}}\right)$ and 
$$\varphi\circ\psi((A,B))\preceq Id_{P(n,k, \vec{s})}((A,B))$$ 
for every $(A,B)\in P(n,k, \vec{s})$. Therefore, by Theorem 4, $\psi\circ\varphi\simeq Id_{Hom\left(K_2, KG(n, k)_{\vec{s}-\textit{stab}}\right)}$ and $\varphi\circ\psi\simeq Id_{P(n, k,\vec{s})}$. Hence, the posets $Hom_p\left(K_2, KG(n, k)_{\vec{s}-\textit{stab}}\right)$ and $P(n, k, \vec{s})$ have the same homotopy type. Now, the proof is completed.
\end{proof}

\section{The proof of the main theorem}

\subsection{The case $\vec{s}=(s_1,\ldots, s_{k-1}, 1)$}
\begin{sloppypar}
In this sub-section we will determine the homotopy type of $P(n, k, \vec{s})$ for ${\vec{s}=(s_1,\ldots, s_{k-1}, 1)}$
where $s_i\geq 2$. In particular, we strengthen the J. Oszt\'{e}nyi result,
the neighborhood complex of $KG(n, k)_{\widetilde{{s-stab}}}$ is homotopy equivalent to $\mathbb{S}^{n-s(k-1)-2}$. It is worth pointing out that Oszt\'{e}nyi's proof was based on discrete Morse theory, the Nerve homotopy theorem, and the following fact. 
\end{sloppypar}  
\begin{theorem}[\cite{lon}]\label{susp}
Let $\mathcal{K}$ be a simplicial complex and $\mathcal{C}$ and $\mathcal{D}$ be contractible subcomplexes such that $\mathcal{K}=\mathcal{C}\cup \mathcal{D}$. Then $\mathcal{K}$ is homotopy equivalent to the suspension of $\mathcal{C}\cap \mathcal{D}$.
\end{theorem}
Here, we use slightly different technique that is mainly based on Order Homotopy Theorem and Theorem \ref{susp}. For simplicity of notation, the symbol $[k,n]$ is used instead of the set $\{k, \ldots, n\}$ where $k\leq n$. 

\begin{theorem}\label{poset1}
For given positive integer $k\geq 2$, an integer vector $\vec{s}=(s_1,\ldots, s_{k-1},1)$ where $s_i\geq 2$, and $n \geq \sum_{i=1}^{k-1} s_i + 2$ the poset
$P(n,k,\vec{s})$ is homotopy equivalent to 
$\mathbb{S}^{n-\sum_{i=1}^{k-1} s_i - 2}$. In particular, $P(n,k,\vec{s})$ is $\left(n-\sum_{i=1}^{k-1} s_i-3\right)$-connected.
\end{theorem}

\begin{proof}
The proof is by induction on $n$. The base case is $n=\sum_{i=1}^{k-1} s_i+2$. Put 
$$U= \{1, s_1+1, s_2+s_1+1, \ldots, \sum_{i=1}^{k-1} s_i+1\}\quad\&\quad V=\{2, s_1+2, s_2+s_1+2, \ldots, \sum_{i=1}^{k-1} s_i+2\},$$
and consider the following subposets of $P(n,k,\vec{s})$:
$$P_1=\{(A, B)\in P(n,k,\vec{s}): V\subseteq A\quad \&\quad  U\subseteq B\}$$
$$P_2=\{(A, B)\in P(n,k,\vec{s}): U\subseteq A \quad \&\quad V\subseteq B\}.$$
It is easy to see that $P(n,k,\vec{s})$ is a disjoint union of $P_1$ and $P_2$.
Moreover, $(V, U)$ and $(U, V)$ are minimal elements of $P_1$ and $P_2$ respectively. Hence, both of them are contractible. Thus, $\Delta\left(P(n,k,\vec{s})\right)\simeq\mathbb{S}^0$.

Now, assume that $n> \sum_{i=1}^{k-1} s_i+2$. Let $P_1$ and $P_2$ be the subposets of $P(n,k,\vec{s})$ defined by
$$P_1=\{(A, B)\in P(n,k,\vec{s}): n\notin B\}$$
$$P_2=\{(A, B)\in P(n,k,\vec{s}): n\notin A\}.$$
Clearly $\Delta\left(P(n,k, \vec{s})\right)= \Delta(P_1)\cup\Delta(P_2)$, and $\Delta(P_1)\cap\Delta(P_2)= \Delta\left(P(n-1,k, \vec{s})\right)$. Thus, by the induction hypothesis $\Delta(P_1)\cap\Delta(P_2)$ is homotopy equivalent to $\mathbb{S}^{n-\sum_{i=1}^{k-1} s_i - 3}$. Now, Theorem \ref{susp} tells us for completing the proof, it suffices to show that $P_1$ and $P_2$ are contractible.
Since $P_1$ is isomorphic with $P_2$, it is enough to consider only the case $P_1$. In order to prove that $P_1$ is contractible, consider the following decreasing chain of subposets of $P_1$:
$$P_1= P^{(0)}\supseteq P^{(1)}\supseteq\cdots\supseteq P^{(k)},$$
where
\begin{align*}
P^{(i)}=\{ & (A, B)\in P_1 : A\cap [n-\left(\sum_{j=1}^{i}s_{k-j}\right)+1, n]=\{n, n-s_{k-1}, \ldots, n-\left(\sum_{j=1}^{i-1}s_{k-j}\right)\}\&\\
& B\cap [n-\left(\sum_{j=1}^{i}s_{k-j}\right), n]=\{n-1, n-s_{k-1}-1, \ldots, n-\left(\sum_{j=1}^{i-1}s_{k-j}\right)-1\}\}.  
\end{align*}
for $1\leq i\leq k$.

We will show that $P_1$ and $P^{(k)}$ have the same homotopy type. This finishes the proof as $P^{(k)}$ is contractible. Indeed, 
$$(\{n, n-s_{k-1}, \ldots, n-\left(\sum_{j=1}^{k-1}s_{k-j}\right)\}, \{n-1, n-s_{k-1}-1, \ldots, n-\left(\sum_{j=1}^{k-1}s_{k-j}\right)-1\})$$
is the minimal element of $P^{(k)}$ which implies $P^{(k)}$ be contractible, and hence $P_1$ is contractible. To confirm our claim, we show that $P^{(i)}\simeq P^{(i+1)}$ for every $0\leq i < k$.  

Define the order-preserving map
 \begin{align*}
\varphi_{1}:P^{(i)} \longrightarrow  P^{(i)} & \\
  \left(A, B\right) &\longmapsto \left(A\cup\{n-\left(\sum_{j=1}^{i}s_{k-j}\right)\}, B\right).
\end{align*}
Note that this function is well-defined as $n-\left(\sum_{j=1}^{i}s_{k-j}\right)\notin B$. Moreover, it is easy to check that $\varphi_1$ is an increasing operator. Thus, by Theorem 4, 
$$P^{(i)}\simeq Img(\varphi_1)=\{(A, B)\in P^{(i)}: n-\left(\sum_{j=1}^{i}s_{k-j}\right)\in A\}.$$
Now, consider the following decreasing operator
 \begin{align*}
\varphi_{2}:Img(\varphi_1) \longrightarrow& Img(\varphi_1)\\
  \left(A, B\right) &\longmapsto \left(A\setminus[n-\left(\sum_{j=1}^{i+1}s_{k-j}\right)+1, n-\left(\sum_{j=1}^{i}s_{k-j}\right)-1], B\right).
\end{align*}
Note that removing the set $[n-\left(\sum_{j=1}^{i+1}s_{k-j}\right)+1, n-\left(\sum_{j=1}^{i}s_{k-j}\right)-1]$ from $A$ does not affect the elements of the maximal $\vec{s}$-stable subset of $A$. Indeed, if 
$$T\subset A\subset [n-\left(\sum_{j=1}^{i}s_{k-j}\right)]\bigcup\{n-\left(\sum_{j=1}^{i-1}s_{k-j}\right), n-\left(\sum_{j=1}^{i-2}s_{k-j}\right), \ldots, n\}$$ is an $\vec{s}$-stable subset of $A$, then this subset has at most one element from 
$$[n-\left(\sum_{j=1}^{i+1}s_{k-j}\right)+1, n-\left(\sum_{j=1}^{i}s_{k-j}\right)].$$
Therefore, 
$$\left(T\setminus [n-\left(\sum_{j=1}^{i+1}s_{k-j}\right)+1, n-\left(\sum_{j=1}^{i}s_{k-j}\right)-1]\right)\bigcup\{n-\left(\sum_{j=1}^{i}s_{k-j}\right)\},$$
is also an $\vec{s}$-stable set. So, the map $\varphi_2$ is well defined. In conclusion, we have 
\begin{align*}
Img(\varphi_1) & \simeq Img(\varphi_2)=\\
& \{(A, B)\in P^{(i)} : A\cap [n-\left(\sum_{j=1}^{i+1}s_{k-j}\right)+1, n-\left(\sum_{j=1}^{i}s_{k-j}\right)]=\{n-\left(\sum_{j=1}^{i}s_{k-j}\right)\}\}.
\end{align*}
Now, the following increasing operator
 \begin{align*}
\varphi_{3}:Img(\varphi_2) \longrightarrow& Img(\varphi_2)\\
  \left(A, B\right) \longmapsto & \left(A, B\cup\{n-\left(\sum_{j=1}^{i}s_{k-j}\right)-1\}\right).
\end{align*}
implies
$$Img(\varphi_2)\simeq Img(\varphi_3)=\{(A, B)\in Img(\varphi_2) :  n-\left(\sum_{j=1}^{i}s_{k-j}\right)-1\in B\}.$$
Finally, consider the following decreasing operator 
 \begin{align*}
\varphi_{4}:Img(\varphi_3) \longrightarrow& Img(\varphi_3)\\
  \left(A, B\right) \longmapsto & \left(A, B\setminus[n-\left(\sum_{j=1}^{i+1}s_{k-j}\right), n-\left(\sum_{j=1}^{i}s_{k-j}\right)-2]\right).
\end{align*}
With the same reason as discussed for the map $\varphi_2$, the map $\psi_1$ is well-defined. Thus,
\begin{align*}
Img(\varphi_3)& \simeq Img(\varphi_4)= \{ (A, B)\in P^{(i)} :\\
& A\cap [n-\left(\sum_{j=1}^{i+1}s_{k-j}\right)+1, n-\left(\sum_{j=1}^{i}s_{k-j}\right)]=\{n-\left(\sum_{j=1}^{i}s_{k-j}\right)\}, \\
& B\cap [n-\left(\sum_{j=1}^{i+1}s_{k-j}\right),n-\left(\sum_{j=1}^{i}s_{k-j}\right)-1]=\{n-\left(\sum_{j=1}^{i}s_{k-j}\right)-1\}\}=P^{(i+1)}.  
\end{align*}
Now, the proof is completed.
\end{proof}

\subsection{The case $\vec{s}=(s_1,\ldots, s_{k-1}, 2)$}
This sub-section is devoted to determining the homotopy type of $P(n, k, \vec{s})$ for $\vec{s}=(s_1,\ldots, s_{k-1}, 2)$ where $s_i\geq 2$.
Our main tool for this purpose is Discrete Morse Theory. 

The Discrete Morse Theory, introduced by Forman \cite{F98}, is a convenient tool in combinatorial topology for proving homotopy equivalences. It works on the face poset $P=P(\mathcal{K})$ of a simplicial complex $\mathcal{K}$. A {\it partial matching} on $P$ is a set $\Sigma\subseteq P$, and an injective map $\mu:\Sigma\rightarrow P\setminus\Sigma$, such that $\mu(\sigma)\succ \sigma$ ($\succ$ is the lower cover relation on $P$), for all $\sigma\in\Sigma$. The elements of  $P\setminus(\Sigma\cup\mu(\Sigma))$ are called {\it critical}. Additionally, such a partial matching $\mu$ is called {\it acyclic} if there does not exist a sequence of distinct elements $\sigma_1,\dots,\sigma_t\in\Sigma$, where $t\geq 2$, satisfying $\mu(\sigma_1)\succ \sigma_2$, $\mu(\sigma_2)\succ \sigma_3$,$\dots$, $\mu(\sigma_k)\succ \sigma_1$. We will use the main theorem of DMT:

\begin{theorem}[\,{\cite[Theorem 11.13]{K08}}\,]\label{DMT}
	Let $\mathcal{K}$ be a simplicial complex, and let $\mu$ be an acyclic matching on the face poset of $\mathcal{K}$. If the critical simplices form a sub-complex $\mathcal{K}_c$ of $\mathcal{K}$, then $\mathcal{K}_c$ and $\mathcal{K}$ are homotopy equivalent.
\end{theorem}
 
 Bj\"{o}rner and de Longueville \cite{lon}, when studied the neighborhood complex of $KG(n,k)_{2-stab}$, verified that a subcomplex of the neighborhood complex of the almost 2-stable Kneser graph is homotopy equivalent to the neighborhood complex of the 2-stable Kneser graph. Now we prove a generalization of this result. 
  
 \begin{theorem}\label{poset2}
 	For given positive integer $k\geq 2$, and integer vectors $\vec{s}=(s_1,\ldots, s_{k-1},1)$ and $\vec{s}_\ast=(s_1,\ldots, s_{k-1},2)$ where $s_i\geq 2$, and $n \geq \sum_{i=1}^{k-1} s_i + 2$ the complex $\Delta(P(n,k,\vec{s}))$ is homotopy equivalent to $\Delta(P(n,k,\vec{s}_\ast))$.
 \end{theorem}
 
 \begin{proof}
\begin{sloppypar}
 The complex $\Delta(P(n,k,\vec{s}_\ast))$ is a subcomplex of $\Delta(P(n,k,\vec{s}))$. A simplex $\sigma=\langle (A_1, B_1)\subset\cdots\subset (A_l, B_l)\rangle$ of $\Delta(P(n,k,\vec{s}))$ is in $\Delta(P(n,k,\vec{s}))\setminus\Delta(P(n,k,\vec{s}_\ast))$ if $A_1$ or $B_1$ does not contain $\vec{s}_\ast$-stable $k$-set. We will collapse all those simplex in more steps. To describe this simplicial collapsing we will use the Discrete Morse Theory.
\end{sloppypar}
 Let $H_1$ denote the set of simplices $\sigma=\langle (A_1, B_1)\subset\cdots\subset (A_l, B_l)\rangle$ of $\Delta(P(n,k,\vec{s}))$ that $[n]\setminus A_l$ does not contain $\vec{s}_\ast$-stable $k$-set. For a simplex $\sigma$ of $H_1$ let $C(\sigma)$ be the lexicographically smallest $\vec{s}$-stable $k$-set in $[n]\setminus A_l$ and
 	\[
 	r(\sigma):=\max\{\{ j: C(\sigma)\not\subseteq B_j\}\cup\{0\}\},
 	\]
 	\[
 	q(\sigma):=\max\{\{ j: B_j=C(\sigma)\}\cup\{0\}\}.
 	\]
 	Now we define the partial matching $(\mu_1,\Sigma_1)$ on $\Delta(P(n,k,\vec{s}))$: let 
 	\[
 	\begin{array}{l}
 	\Sigma_{1,1}:=\{\sigma\in H_1 \colon r(\sigma)=l(\sigma)\}; \\
 	\Sigma_{1,2}:=\{\sigma\in H_1 \colon 0<r=r(\sigma)< l(\sigma) \mbox{ and } (A_{r+1},B_{r+1})\neq (A_r, B_r\cup C(\sigma))\};\\
 	\Sigma_{1,3}:=\{\sigma\in H_1 \colon r(\sigma)=0 \mbox{ and } q(\sigma)=0\};\\
 	\Sigma_{1,4}:=\{\sigma\in H_1 \colon r(\sigma)=0 \mbox{ and } 0<q=q(\sigma)< l(\sigma) \mbox{ and } A_{q+1}\neq A_q\};
 	\end{array} 
 	\]
 	\[
 	\mbox{and } \Sigma_1:=\Sigma_{1,1}\cup\Sigma_{1,2}\cup\Sigma_{1,3}\cup\Sigma_{1,4},
 	\]
 	where $l(\sigma)$ is the index of the maximal element of the chain $\sigma$.
 	For $\sigma\in\Sigma_1$ we define:\\
 	$\mu_1(\sigma):= $
 	\[
 	\left\{\begin{array}{ll}
 	\langle (A_1,B_1)\subset\cdots\subset (A_l,B_l)\subset (A^\star,B^\star) \rangle, & \mbox{if } \sigma\in\Sigma_{1,1},\\
 	\langle\cdots\subset (A_r,B_r)\subset (A^\star, B^\star) \subset (A_{r+1}, B_{r+1}) \subset\cdots\rangle, & \mbox{if } \sigma\in\Sigma_{1,2},\\
 	\langle (A^\circ,B^\circ) \subset (A_1,B_1)\subset\cdots\subset (A_l,B_l) \rangle, & \mbox{if } \sigma\in\Sigma_{1,3},\\
 	\langle \cdots\subset (A_q,B_q)\subset (A^\circ, B^\circ) \subset (A_{q+1},B_{q+1}) \subset\cdots\rangle, & \mbox{if } \sigma\in\Sigma_{1,4}.\\
 	\end{array} \right.
 	\]
 	where $A^\star=A_r$ and $B^\star=B_r\cup C(\sigma)$ for a chain $\sigma\in \Sigma_{1,1}\cup \Sigma_{1,2}$, and $A^\circ=A_{q+1}$ and $B^\circ=C(\sigma)$ for a chain $\sigma\in \Sigma_{1,3}\cup\Sigma_{1,4}$. 
 	
 	It is easy to see that $\mu_1(\Sigma_1)\subseteq H_1$, namely $\mu_1$ add a new element to the chain $\sigma\in \Sigma_1$ that $A_{l(\sigma)}=A_{l(\mu_1(\sigma))}$. Furthermore, $\mu_1$ is injective and $\Sigma_1\cap\mu_1(\Sigma_1)=\emptyset$: let $\tau\in\mu_1(\Sigma_1)$ and suppose that $\mu_1(\sigma)=\tau$. We have to delete an element from $\tau$ to get $\sigma$. If $0<r(\tau)< l(\tau)$, then $(A_{r+1},B_{r+1})= (A_r, B_r\cup C(\tau))$. So $\tau\notin\Sigma_1$ and $\tau =\mu_1(\tau\setminus\{(A_{r+1},B_{r+1})\})$.
 	If $r(\tau)=0$ and $0<q(\tau)<l(\tau)$, then $A_{q+1}= A_q$. So $\tau\notin\Sigma_1$ and $\tau= \mu_1(\tau\setminus\{(A_{q},B_{q})\})$.
 	If we delete a distinct element from $(A_{r+1},B_{r+1})$ and $(A_{q},B_{q})$, then $\sigma\notin\Sigma_1$ or $\mu_1(\sigma)\neq\tau$.  
 \begin{sloppypar}
 Now we show that this matching is acyclic. By contradiction assume that $\sigma_1, \mu_1(\sigma_1),\dots,\sigma_t, \mu_1(\sigma_t)$ is a cycle ($t>1$). In each up-step, $\sigma_i\prec\mu_1(\sigma_i)$ we add a new point to the chain $\sigma_i$ such that $A_{l(\mu_1(\sigma_i))}=A_{l(\sigma_i)}$ and so $C(\mu_1(\sigma_i))=C(\sigma_i)$. In each down-step, $\mu_1(\sigma_i)\succ\sigma_{i+1}$ we delete the point $(A_r,B_r)$ or $(A_{q+1},B_{q+1})$ or $(A_l,B_l)$ from the chain $\mu_1(\sigma_i)$, otherwise $\sigma_{i+1}\not\in \Sigma_1$. If we delete the point $(A_r, B_r)$ from $\mu_1(\sigma_i)$, then $A_{l(\sigma_{i+1})}=A_{l(\mu_1(\sigma_i))}$ and so $C(\sigma_{i+1})=C(\mu_1(\sigma_i))$. Similarly, if we delete the point $(A_{q+1},B_{q+1})$ from $\mu_1(\sigma_i)$, then $A_{l(\sigma_{i+1})}=A_{l(\mu_1(\sigma_i))}$ and so $C(\sigma_{i+1})=C(\mu_1(\sigma_i))$. Else if we delete the point $(A_l,B_l)$ from $\mu_1(\sigma_i)$, then $A_{l(\sigma_{i+1})}\subset A_{l(\mu_1(\sigma_i))}$ and $C(\sigma_{i+1})$ is lexicographically smaller than $C(\mu_1(\sigma_i))$, otherwise $\sigma_{i+1}\not\in \Sigma_1$. That is $A_{l(\sigma_i)}$ and so $C(\sigma_i)$ doesn't change in the cycle or $A_{l(\sigma_i)}$ descend. If $C(\sigma_i)$ doesn't change then the index $r(\sigma_i)$ decrease or the index $q(\sigma_i)$ increase in the cycle, so the cycle isn't possible. If $A_{l(\sigma_i)}$ descend, then $\mu_1(\sigma_t)\succ\sigma_1$ isn't possible, so this is a contradiction again. Therefore $\mu_1$ is an acyclic partial matching.
\end{sloppypar}
 	
 	Let $H_2$ denote the subset $\{ \sigma\in H_1 \colon q(\sigma)=l(\sigma)\}$ of $H_1$. The critical simplices of $\mu_1$ form the complex $\Delta(P_{\vec{s}}(n,k,2))\setminus (H_1\setminus H_2)$. Using the main theorem of DMT we get that
 	\[
 	\Delta(P(n,k,\vec{s}))\simeq \Delta(P(n,k,\vec{s}))\setminus (H_1\setminus H_2).
 	\]
 	
 	Next, we eliminate the rest of the simplices of $H_1$. For a simplex $\sigma$ of $H_2$ let $D(\sigma)$ be the lexicographically smallest $\vec{s}_\ast$-stable $k$-set in $A_1$ and $E(\sigma)$ be the leftside neighbor of $D(\sigma)$, i.e., if $D(\sigma)=\{D(\sigma)(1),D(\sigma)(2),\dots ,D(\sigma)(k)\}$ then $E(\sigma):=\{D(\sigma)(1)-1,D(\sigma)(2)-1,\dots ,D(\sigma)(k)-1\}$. The condition that $B_1$ is an $\vec{s}$-stable $k$-set assures that $E(\sigma)$ is an $\vec{s}_\ast$-stable $k$-set. Let $e(\sigma):=\max\{\{ j: A_j\cap E(\sigma)=\emptyset \}\cup\{0\}\}$. Now we are ready to define the partial matching $(\mu_2,\Sigma_2)$ on $\Delta(P(n,k,\vec{s}))\setminus (H_1\setminus H_2)$: let 
 	\[
 	\Sigma_2:=\{\sigma\in H_2 \colon e(\sigma)=0, \mbox{ or } A_{e+1}\neq A_e\cup E(\sigma)\}
 	\]
 	and\\
 	$\mu_2(\sigma):= $
 	\[
 	\left\{\begin{array}{ll}
 	\langle (A^\diamond,B^\diamond) \subset (A_1,B_1)\subset\cdots\subset (A_l,B_l) \rangle, & \mbox{if } e(\sigma)=0,\\
 	\langle \cdots\subset (A_e,B_e)\subset (A^\diamond , B^\diamond) \subset (A_{e+1},B_{e+1}) \subset\cdots\rangle, & \mbox{otherwise.}\\
 	\end{array} \right.
 	\]
 	where $A^\diamond=A_{e+1}\setminus E(\sigma)$ and $B^\diamond=B_{e+1}$ for a chain $\sigma\in \Sigma_2$. 
 	
 	It is easy to see that $\Sigma_2\cap\mu_2(\Sigma_2)=\emptyset$, $\Sigma_2\cup\mu_2(\Sigma_2)=H_2$ and $\mu_2$ is injective.
\begin{sloppypar}
 We show that this matching is acyclic. By contradiction assume that $\sigma_1, \mu_2(\sigma_1),\dots,\sigma_t, \mu_2(\sigma_t)$ is a cycle ($t>1$). In each upstep, $\sigma_i\prec\mu_2(\sigma_i)$ we add a new point to the chain $\sigma_i$ such that $D(\mu_2(\sigma_i))=D(\sigma_i)$. In each downstep, $\mu_2(\sigma_i)\succ\sigma_{i+1}$ we delete the point $(A_e,B_e)$ or $(A_1,B_1)$ from the chain $\mu_2(\sigma_i)$, otherwise $\sigma_{i+1}\not\in \Sigma_2$. If we delete the point $(A_e, B_e)$ from $\mu_2(\sigma_i)$, then  $D(\sigma_{i+1})=D(\mu_1(\sigma_i))$. If we delete the point $(A_1,B_1)$ from $\mu_2(\sigma_i)$, then $D(\sigma_{i+1})$ is lexicographically smaller than $D(\mu_1(\sigma_i))$, otherwise $\sigma_{i+1}\not\in \Sigma_2$. That is  $D(\sigma_i)$ doesn't change in the cycle or $A_1$ ascend. If $D(\sigma_i)$ doesn't change then the index $e(\sigma_i)$ decrease in the cycle, so the cycle isn't possible. If $A_1$ ascend, then $\mu_1(\sigma_t)\succ\sigma_1$ isn't possible, so this is a contradiction again.
\end{sloppypar}
 	
 	Therefore $\mu_2$ is an acyclic partial matching, and the critical simplices form the complex $\Delta(P_{\vec{s}}(n,k,2))\setminus H_1$. Using the main theorem of DMT we get that
 	\[
 	\Delta(P(n,k,\vec{s}))\simeq \Delta(P(n,k,\vec{s}))\setminus H_1.
 	\]
 	
 	Let $H_3$ be the set of simplices $\sigma$ of $\Delta(P(n,k,\vec{s}))$ that $[n]\setminus B_l$ does not contain $\vec{s}_\ast$-stable $k$-set. We note that $H_1\cap H_3= \emptyset$, so $H_3$ is the set of the simplex of $\Delta(P_{\vec{s}}(n,k,2))\setminus H_1$. Similarly as above, we can collapse all simplex $\sigma\in H_3$ from $\Delta(P(n,k,\vec{s}))\setminus H_1$ by two partial matchings. That is 
 	
 	\[
 	\Delta(P(n,k,\vec{s}))\simeq (\Delta(P(n,k,\vec{s})))\setminus (H_1\cup H_3)).
 	\]

 	Next, let $H_5$ be the set of the simplex $\sigma$ of $\Delta(P(n,k,\vec{s}))\setminus (H_1\cup H_3)$ that $B_1$ does not contain $\vec{s}_\ast$-stable $k$-set. For a simplex $\sigma$ of $H_5$ let $F(\sigma)$ be the lexicographically smallest $\vec{s}_\ast$-stable $k$-set in $[n]\setminus A_l$ and $f(\sigma):=\max\{j: f(\sigma)\not\subseteq B_j\}$. This set exists, because we collapsed all $\sigma$ such that $[n]\setminus A_l$ does not contain $\vec{s}_\ast$-stable $k$-set, and $F(\sigma)\not\subseteq B_1$ so $1\leq f$. The simplices of $H_5$ are collapsible by the following partial matching: 
 	\[
 	\Sigma_5:=\{\sigma\in H_5 \colon f(\sigma)=l(\sigma)\}\cup
 	\]
 	\[
 	\{\sigma\in H_5 \colon f=f(\sigma)< l(\sigma) \mbox{ and } (A_{f+1},B_{f+1})\neq (A_f,B_f\cup D(\sigma))\}
 	\]
 	and $\mu_5(\sigma)= $
 	\[
 	\left\{\begin{array}{l}
 	\langle (A_1,B_1)\subset\cdots\subset (A_l,B_l)\subset (A^\diamond,B^\diamond) \rangle,\mbox{ if } f(\sigma)=l(\sigma),\\
 	\langle \cdots\subset (A_f,B_f)\subset (A^\diamond,B^\diamond) \subset (A_{f+1},B_{f+1}) \subset\cdots\rangle,\ \mbox{otherwise},\\
 	\end{array} \right.
 	\]
 	where $A^\diamond=A_f$ and $B^\diamond=B_f\cup \{F(\sigma)\}$ for a chain $\sigma\in \Sigma_5$. 
 	
 	It can be checked that (as above)
 	\begin{itemize}
 		\item $\Sigma_5 \cap\mu_5(\Sigma_5)=\emptyset$;
 		\item $\Sigma_5 \cup\mu_5(\Sigma_5)=H_5$;
 		\item $\mu_5$ is injective;  
 		\item $\mu_5$ is an acyclic matching;
 		\item the critical simplices form the complex $$\Delta(P(n,k,\vec{s}))\setminus (H_1\cup H_3\cup H_5).$$
 	\end{itemize}
 	
 	Using the main theorem of DMT we get that
 	\[
 	\Delta(P(n,k,\vec{s}))\simeq \Delta(P(n,k,\vec{s}))\setminus (H_1\cup H_2)\simeq \Delta(P(n,k,\vec{s}))\setminus (H_1\cup H_3\cup H_5).
 	\]
 	
 	Let $H_6$ be the set of the simplex $\sigma$ of $\Delta(P(n,k,\vec{s}))\setminus (H_1\cup H_3)$ that $A_1$ does not contain $\vec{s}_\ast$-stable $k$-set. Similarly, we have $H_5\cap H_6=\emptyset$, and we can collapse all simplex $\sigma\in H_6$ from $\Delta(P(n,k,\vec{s}))\setminus (H_1\cup H_3\cup H_5)$ by a partial matching. That is 
 	\[
 	\Delta(P(n,k,\vec{s}))\simeq \Delta(P(n,k,\vec{s}))\setminus (H_1\cup H_3\cup H_5)\simeq
 	\]
 	\[
 	\Delta(P(n,k,\vec{s}))\setminus (H_1\cup H_3\cup H_5\cup H_6)= \Delta(P(n,k,\vec{s}_\ast))
 	\]  
 	
 \end{proof}
Now, Theorems \ref{poset1}, \ref{poset2}, and Lemma \ref{neighborhoodcomplex} imply Theorem \ref{main}

\section{The chromatic number of $KG(n, k)_{\vec{s}-stab}$}
The proof of Theorem \ref{chromaticnumber} is a simple corollary of Theorem \ref {main} and the following observation. There is a proper coloring of $KG(n, k)_{\vec{s}-stab}$ with $n-\left(\sum_{i=1}^{k-1}s_i\right)$-number of colors. Indeed, it is not hard to check that the following map
\begin{align*}
c : {\binom{[n]}{k}}_{\vec{s}-stab} &\longrightarrow [n-(\sum_{i=1}^{k-1}s_i)]\\
A &\longmapsto  \min\{\min A, n-(\sum_{i=1}^{k-1}s_i)\},
\end{align*}
defines a proper coloring of $KG(n, k)_{\vec{s}-stab}$. Thus, $\chi\left(KG(n, k)_{\vec{s}-stab}\right)\leq n-\left(\sum_{i=1}^{k-1}s_i\right)$. On the other hand, combining Lov\'{a}sz lower bound with Theorem \ref{main}, and the fact that any $d$-sphere is $(d-1)$-connected imply that
$\chi\left(KG(n, k)_{\vec{s}-stab}\right)\geq n-\left(\sum_{i=1}^{k-1}s_i\right)$ for the mentioned parameters in Theorem \ref{main}. Hence, 
$$\chi\left(KG(n, k)_{\vec{s}-stab}\right)= n-(\sum_{i=1}^{k-1}s_i),$$
where $k\geq 2$, $n\geq\sum_{i=1}^{k-1}s_i+2$, $s_i\geq2$ for $i\neq k$ and $s_k\in\{1,2\}$. As a simple corollary of Theorem \ref{chromaticnumber}, we determine the chromatic number of $3$-stable Kneser graphs with at most one error. 
\begin{corollary}
	Let $n,k$ be positive integers with $n\geq 3k$. Then
	$$\chi\left({KG(n,k)}_{3-stab}\right)\geq n-3(k-1)-1.$$
\end{corollary}
\begin{proof}
	This is an easy task as the graph ${KG(n-1,k)}_{(3,\ldots,3,2)}$ is a sub-graph of ${KG(n,k)}_{3-\textit{stab}}$ which implies
	$$n-3(k-1)-1=\chi\left({KG(n-1,k)}_{(3,\ldots,3,2)}\right)\leq\chi\left({KG(n,k)}_{3-\textit{stab}}\right).$$
\end{proof}

One may speculate that one may gives a sharper lower bound on the chromatic number of $3$-stable Kneser graphs by determining the homotopy type of $\mathcal{N}\left(KG(n,k)_{3-stab}\right)$. Unfortunately, this is not the case. Oszt\'{e}nyi has seen in \cite{osz} that $\mathcal{N}\left(KG(n,k)_{3-stab}\right)$ is not $(n-3k)$-connected in general. Finally, regrading Theorem \ref{chromaticnumber} and Conjecture \ref{conj}, the following question might be interesting.
\begin{question}
What is the chromatic number of ${KG(n, k)}_{\vec{s}-\textit{stab}}$ for an arbitrary integer vector $\vec{s}=(s_1,\ldots, s_k)$ and $n\geq\sum_{i=1}^{k}s_i$?
\end{question}

\section*{Acknowledgments}
The research of the second author is supported by \textbf{ EFOP-3.6.1-16-2016-00006 "The development and enhancement of the research potential at John von Neumann University"} project. The Project is supported by the Hungarian Government and co-financed by the European Social Fund.

\end{document}